\documentclass{amsart}
\newtheorem{theorem}{Theorem}[section]
\newtheorem{lemma}[theorem]{Lemma}

\newtheorem{proposition}[theorem]{Proposition}
\theoremstyle{definition}

\theoremstyle{remark}

\numberwithin{equation}{section}

\begin{document}

\title{Explicit  coproduct formula\\ for quantum group of the type $G_2$}
\author{V.K. Kharchenko, C. Vay}
\address{Universidad Nacional Aut\'onoma de M\'exico, Facultad de Estudios Superiores 
Cuautitl\'an, Primero de Mayo s/n, Campo 1, CIT, Cuautitlan Izcalli, Edstado de M\'exico, 54768, MEXICO.} 
\address{FAMAF-CIEM (CONICET), Universidad Nacional de C\'ordoba, Medina Allegre s/n, 
Cuidad Universitaria, 5000 C\'ordoba, Rep\'ublica Argentina.}
\email{vlad@unam.mx; vay@famaf.unc.edu.ar}
\thanks{The authors was supported by PAPIIT IN 112913, UNAM, and PACIVE CONS-18, FES-C UNAM, M\'exico. 
This work was carried out in part during the visit of the second author to UNAM. 
He would like to thank the first author and Viqui Lara Sagahon for their worm hospitality and support.}
  


\begin{abstract}
We find a coproduct formula in the explicit form for PBW-generators 
of the two-parameter quantum group  $U_q^+(\frak{g})$ where $\frak{g}$ is a simple Lie algebra of type
$G_2.$ The similar formulas for quantizations of simple Lie algebras of infinite series are already known.
\end{abstract}
\maketitle
\markboth{V.K. Kharchenko, C. Vay}{Coproduct formula}

\section{introduction}
In the present paper, we establish an explicit  coproduct formula 
for the quantum group $U_q(\mathfrak{g}),$ where  $\mathfrak{g}$ is a simple Lie algebra of type $G_2.$
The explicit formulas related to the Weyl basis of simple Lie algebras of infinite series appear 
in \cite[Lemma 6.5]{AS}, \cite[Lemma 3.5]{KA}, \cite[Theorem 4.3]{Kh11} and \cite{Kh15}.

The Weyl basis of $\mathfrak{g}$ of the type $G_2$ consists of the polynomials
\begin{equation}
x_1, x_2, [x_1,x_2], \ [[x_1,x_2],x_2], \ [[[x_1,x_2],x_2],x_2],\  [[[[x_1,x_2],x_2],x_2], x_1],
\label{n2}
\end{equation}
that correspond to the positive roots
$$
\alpha_1, \alpha_2, \alpha_1+\alpha_2, \alpha_1+2\alpha_2, \alpha_1+3\alpha_2, 2\alpha_1+3\alpha_2,
$$
see \cite[Chapter VI, \S 4]{Ser} or  \cite[Chapter IV, \S 3, XVII]{Jac}.
These polynomials define some elements of  $U_q^+(\mathfrak{g})$ if the Lie operation is replaced with 
the skew brackets. The coproduct of all of them, except the last one, may be found by the general formula 
proved in Proposition \ref{2coSer}:
$$
\Delta ([x_1x_2^n])=[x_1x_2^n]\otimes 1 +\sum _{k=0}^n\alpha _k^{(n)}g_1g_2^{n-k}\, x_2^k\otimes [x_1x_2^{n-k}], 
$$
where 
$$
\alpha _k^{(n)}=\hbox{\rm \huge [}^n  _k\hbox{\rm \huge ]}_{p_{22}} \cdot \prod_{s=n-k}^{n-1}(1-p_{12}p_{21}p_{22}^s),
$$
and  $p_{ij}$ are the quatization parameters that define the commutation rules beetwen the group-like elements 
and the generators: $x_ig_j=p_{ij}g_jx_i.$
This formula takes more elegant form if we consider proportional elements
$$
\{x^n\}={x^n\over [n]_q!}, \ \ \ \{xy^n\}={[\ldots [[x,y],y], \ldots , y]\over 
[n]_q! \cdot \prod_{s=0}^{n-1}(1-p_{12}p_{21}p_{22}^s)}.
$$

For these elements, the coproduct obeys  the following form:
$$
\Delta (\{x_2^n\} )=\sum _{k=0}^n g_2^{n-k}\, \{x_2^{k}\} \otimes \{x_2^{n-k}\} , 
\label{Ser3}
$$
$$
\Delta (\{x_1x_2^n\} )=\{x_1x_2^n\} \otimes 1 +\sum _{k=0}^n g_1g_2^{n-k}\, \{x_2^k\}\otimes \{x_1x_2^{n-k}\} . 
$$
Clearly, the above formulas make sense only if $[k]_q\neq 0,$ $p_{12}p_{21}\neq q^{1-k},$ $q=p_{22}$ for all $k,$
$1\leq k\leq n$ because otherwise the elements  $\{x_2^n\} $ and $\{x_1x_2^n\}$ are undefined.
A similar formula is valid for the elements with opposite alignment of brackets  
$$
\Delta (\{x_2^nx_1\} )=\{x_2^nx_1\} \otimes 1 +\sum _{k=0}^n g_2^k\,  \{x_2^{n-k}x_1\} \otimes  \{x_2^k\} ,
$$
where by definition
$$
\{x_2^nx_1\}={[x_2,[x_2,\ldots [x_2,x_1]\ldots ]] 
\over [n]_q! \cdot \prod_{s=0}^{n-1}(1-p_{12}p_{21}p_{22}^s)}.
$$
The coproduct of the remaining in (\ref{n2}) element has no elegant form. 
Nevertheless, if we replace it  with   
$$
\{ x_1x_2^3x_1\} \stackrel{df}{=} p_{21}^2{q^3+q^2\over 1-q^3}\{ x_1x_2\} \{ x_1x_2^2\}
+p_{21}{[4]_q-2\over 1-q^3}\{ x_1x_2^2\} \{ x_1x_2\}+\{ x_1x_2^3\}x_1,
$$
then the ordered set 
$$
x_2<\{ x_1x_2^3\}<\{ x_1x_2^2\} < \{ x_1x_2^3x_1\}<\{ x_1x_2\}<x_1
$$
forms a set of PBW generators of the algebra $U_q^+({\frak g})$ over $G$ (Prorosition \ref{tion}),
and a harmonic formula is valid  (Theorem \ref{c5}):
$$
\Delta (\{ x_1x_2^3x_1\})=\{ x_1x_2^3x_1\} \otimes 1 +g_1^2g_2^3\otimes \{ x_1x_2^3x_1\} 
+\sum _{k=0}^3 g_1g_2^k\,  \{x_2^{3-k}x_1\} \otimes  \{x_1x_2^k\} .
$$
Of course, the set of PBW generators for  $U_q(\mathfrak{g})$ is the union of that sets for
positive and negative quantum Borel subalgebras. By this reason we 
consider only the positive quantum Borel subalgebra $U_q^+(\mathfrak{g}).$

\section{Coproduct of Serre polynomials}
Let $X=$ $\{ x_1, x_2,\ldots, x_n\} $ be a set of quantum variables; that is, 
associated with each $x_i$ are an element $g_i$ of a fixed Abelian group 
$G$ and a character $\chi ^i:G\rightarrow {\bf k}^*.$ 
For every word $w$ in $X$ let $g_w$ or gr$(w)$ denote
an element of $G$ that appears from $w$ by replacing each $x_i$ with $g_i.$
In the same way  $\chi ^w$ denotes a character that appears from $w$ by replacing each $x_i$ with $\chi ^i.$

Let $G\langle X\rangle $ denote the skew group algebra generated by $G$
and {\bf k}$\langle X\rangle $ with the commutation rules $x_ig=\chi ^i(g)gx_i,$
or equivalently $wg=\chi ^w(g)gw,$ where $w$ is an arbitrary word in $X.$
If $u,$ $v$ are homogeneous polynomials in  each $x_i,$ $1\leq i\leq n$,
then the skew brackets are defined by the formula 
\begin{equation}
[u,v]=uv-\chi ^u(g_v) vu.
\label{sqo}
\end{equation}
We use the notation  $\chi ^u(g_v)=p_{uv}=p(u,v).$
The form $p(\hbox{-},\hbox{-})$ is bimultiplicative: 
\begin{equation}
p(u, vt)=p(u,v)p(u,t), \ \ p(ut,v)=p(u,v)p(t,v).
\label{sqot}
\end{equation}
In particular $p(\hbox{-},\hbox{-})$ is completely defined by $n^2$ parameters $p_{ij}=\chi ^{i}(g_{j}).$ 

The algebra $G\langle X\rangle $ has a Hopf algebra structure given by the comultiplications on the generators:
$$
\Delta (x_i)=x_i\otimes 1+g_i\otimes x_i, \ 1\leq i\leq n, \ \ \Delta (g)=g\otimes g, \ g\in G.
$$
The so-called $q$-Serre (noncommutative) polynomials, 
$$[\ldots [x_i,\underbrace{x_j],x_j],\ldots ,x_j]}_m\stackrel{df}{=}[x_ix_j^m], \ \ 1\leq i\neq j\leq n,$$
and
$$[\underbrace{x_j,[x_j,\ldots [x_j}_m,x_i]\ldots ]]\stackrel{df}{=}[x_j^mx_i], \ \ 1\leq i\neq j\leq n,$$
are important as the defining relations of the quantizations $U_q^+({\mathfrak g}).$ 
To find the coproduct of that polynomials, we recall the notations of the $q$-combinatoric. 

If $q$ is a fixed parameter, then $[n]_q=1+q+q^2+\cdots +q^{n-1}$ and
$[n]_q!=\prod _{k=1}^n [k]_q.$ The Gauss polynomials are defined as $q$-binomial coefficients,
\begin{equation}
\hbox{\huge [}^n  _k\hbox{\huge ]}_{q}={[n]_q!\over [k]_q!\cdot [n-k]_q!} \, ,
\label{Gau1}
\end{equation} 
that satisfy two $q$-Pascal identities 
\begin{equation}
\hbox{\huge [}^{n+1}  _k\hbox{\huge ]}_{q}=\hbox{\huge [}^n  _{k-1}\hbox{\huge ]}_{q}
+q^k\cdot \hbox{\huge [}^n  _k\hbox{\huge ]}_{q} , \ \ \ \ \ \
\hbox{\huge [}^{n+1}  _k\hbox{\huge ]}_{q}
= \hbox{\huge [}^n  _{k-1}\hbox{\huge ]}_{q}
\cdot q^{n-k+1} +
\hbox{\huge [}^n  _k\hbox{\huge ]}_{q}.
\label{0(Pas1)}
\end{equation}
If $x$ and $y$ are variables subject to the relation $yx=qxy,$ then $q$-Neuton binomial formula is valid
\begin{equation}
(x+y)^n=\sum _{k=0}^n \hbox{\huge [}^n  _k\hbox{\huge ]}_{q}  x^{n-k}y^{k}. 
\label{Neu}
\end{equation}
For example, if we put $x=g_2\otimes x_2,$ $y=x_2\otimes 1,$
then 
$$
yx=x_2g_2\otimes x_2=p_{22}g_2x_2\otimes x_2=qxy, \ \ \ q=p_{22},
$$
and the $q$-Neuton binomial formula implies
\begin{equation}
\Delta (x_2^n)=(x+y)^n=\sum _{k=0}^n \hbox{\huge [}^n  _k\hbox{\huge ]}_{q}g_2^{n-k}\, x_2^{k}\otimes x_2^{n-k}.
\label{mon}
\end{equation}

\begin{proposition} The following explicit coproduct formula is valid:
\begin{equation}
\Delta ([x_1x_2^n])=[x_1x_2^n]\otimes 1 +\sum _{k=0}^n\alpha _k^{(n)}g_1g_2^{n-k}\, x_2^k\otimes [x_1x_2^{n-k}], 
\label{2coSer}
\end{equation}
where 
\begin{equation}
\alpha _k^{(n)}=\hbox{\rm \huge [}^n  _k\hbox{\rm \huge ]}_{p_{22}} \cdot \prod_{s=n-k}^{n-1}(1-p_{12}p_{21}p_{22}^s).
\label{2recko}
\end{equation}
\end{proposition}
\begin{proof}
We shall use induction on $n.$
If $n=0,$ then the equality reduces to 
$\Delta (x_1)=x_1\otimes 1+g_1\otimes x_1,$ whereas $\alpha _0^{(0)}=1.$
Moreover, it is clear that $\alpha _0^{(n)}=1$ for all $n.$ We have,
\begin{equation}
\Delta ([x_1x_2^n])\cdot (x_2\otimes 1)=[x_1x_2^n]x_2\otimes 1
+\sum _{k=0}^n\alpha _k^{(n)}g_1g_2^{n-k}\, x_2^{k+1}\otimes [x_1x_2^{n-k}], 
\label{ar1}
\end{equation}
\begin{equation}
\Delta ([x_1x_2^n])\cdot (g_2\otimes x_2)=[x_1x_2^n]\, g_2\otimes x_2+
\sum _{k=0}^n\alpha _k^{(n)}g_1g_2^{n-k}\, x_2^{k}\, g_2\otimes [x_1x_2^{n-k}]x_2,
\label{ar2}
\end{equation}
\begin{equation} (x_2\otimes 1)\cdot \Delta ([x_1x_2^n])=x_2[x_1x_2^n]\otimes 1+
\sum _{k=0}^n\alpha _k^{(n)}x_2\, g_1g_2^{n-k}\, x_2^{k}\otimes [x_1x_2^{n-k}], 
\label{ar3}
\end{equation}
\begin{equation}
(g_2\otimes x_2)\cdot \Delta ([x_1x_2^n])=g_2[x_1x_2^n]\otimes x_2+
\sum _{k=0}^n\alpha _k^{(n)}g_1g_2^{n-k+1}\, x_2^{k}\otimes x_2[x_1x_2^{n-k}].
\label{ar4}
\end{equation}
In the second and third relations we may move the group-like factors
to the left: 
$$
[x_1x_2^n]g_2=p_{12}p_{22}^n\, g_2[x_1x_2^n],\ \  x_2^k\, g_2=p_{22}^k\, g_2\, x_2^{k}, \ \
x_2\, g_1g_2^{n-k}\, x_2^{k}=p_{21}p_{22}^{n-k} \, g_1g_2^{n-k+1}\, x_2^{k+1}.
$$

Using all that relations, we develop the coproduct  of
$$[x_1x_2^{n+1}]=[x_1x_2^n]x_2-p_{12}p_{22}^n\, x_2[x_1x_2^n]$$
taking into account that $\Delta (x_2)=x_2\otimes 1+g_2\otimes x_2.$
The sums of (\ref{ar1}) and (\ref{ar3}) provide the tensors
$$
[x_1x_2^{n+1}]\otimes 1+\sum _{k=0}^n\alpha _k^{(n)}(1-p_{12}p_{21}p_{22}^{2n-k})g_1g_2^{n-k}\, x_2^{k+1}\otimes [x_1x_2^{n-k}] ,
$$
whereas the sums of  (\ref{ar2}) and (\ref{ar4}) produce the following ones:
$$
\sum _{k=0}^n\alpha _k^{(n)}p_{22}^k\, g_1g_2^{n-k+1}\, x_2^{k}\otimes [x_1x_2^{n-k+1}].
$$
The first term of (\ref{ar2}) cancels with the first term of (\ref{ar4}). Finally, we arrive to the formula (\ref{2coSer})
with $n\leftarrow n+1$ and coefficients 
\begin{equation}
\alpha _k^{(n+1)}=\alpha _{k-1}^{(n)} \, (1-p_{12}p_{21}p_{22}^{2n-k+1})
+\alpha _{k}^{(n)}\, p_{22}^k, \ \ k\geq 1, \ \ \alpha _0^{(n+1)}=1.
\label{2rec}
\end{equation}
To prove the coproduct formula (\ref{2coSer}), it remains to check that values (\ref{2recko})
satisfy the above recurrence relations. 

To this end, we shall check the equality of the following
two polynomials in commutative variables  $\lambda , q:$
\begin{equation}
\hbox{\huge [}^{n+1}  _k\hbox{\huge ]}_{q} 
\cdot (1-\lambda q^n)=
\hbox{\huge [}^n  _{k-1}\hbox{\huge ]}_{q}
\cdot (1-\lambda q^{2n-k+1})+
\hbox{\huge [}^n  _k\hbox{\huge ]}_{q}
\cdot (1-\lambda q^{n-k})\cdot q^k.
\label{2pol}
\end{equation}
If $\lambda =0,$ then the equality reduces to the first $q$-Pascal identity (\ref{0(Pas1)}).
Let us compare the coefficients at $\lambda ,$ 
$$
\hbox{\huge [}^{n+1}  _k\hbox{\huge ]}_{q}
\cdot q^n=
\hbox{\huge [}^n  _{k-1}\hbox{\huge ]}_{q}
 \cdot q^{2n-k+1}+
\hbox{\huge [}^n  _k\hbox{\huge ]}_{q} \cdot q^{n-k}\cdot q^k.
$$
This equality differs from the second $q$-Pascal identity (\ref{0(Pas1)})
just by a common factor $q^n.$ Hence, the equality (\ref{2pol}) is valid.

If we multiply both sides of (\ref{2pol})
by $\prod _{s=n-k+1}^{n-1}(1-\lambda q^s)$ and next replace the variables
 $q\leftarrow p_{22},$ $\lambda \leftarrow p_{12}p_{21},$  then we obtain precisely (\ref{2rec})
for values (\ref{2recko}).  
\end{proof}

One may illiminate all coeficients in these coproduct formulas replasing the elements 
by some scalar multiples of them. Let us put
\begin{equation}
\{x^n\}={x^n\over [n]_q!}, \ \ \ \{xy^n\}={[xy^n]\over [n]_q! \cdot \prod_{s=0}^{n-1}(1-p_{12}p_{21}p_{22}^s)}.
\label{ska1}
\end{equation}
Then the coproduct formulas obey more elegant form:
\begin{equation}
\Delta (\{x_2^n\} )=\sum _{k=0}^n g_2^{n-k}\, \{x_2^{k}\} \otimes \{x_2^{n-k}\} ,
\label{Ser3}
\end{equation}
\begin{equation}
\Delta (\{x_1x_2^n\} )=\{x_1x_2^n\} \otimes 1 +\sum _{k=0}^n g_1g_2^{n-k}\, \{x_2^k\}\otimes \{x_1x_2^{n-k}\} . 
\label{mon1} 
\end{equation}
Of course, the above formulas make sense only if $[k]_q\neq 0,$ $p_{12}p_{21}\neq q^{1-k},$ $q=p_{22}$ for all $k,$
$1\leq k\leq n$ because otherwise the elements  $\{x_2^n\} $ and $\{x_1x_2^n\}$ are undefined.

In perfect analogy one may develop a coproduct formula for the elements 
$$
[\underbrace{x_2,[x_2,\ldots [x_2}_n,x_1]\ldots ]]\stackrel{df}{=}[x_2^nx_1]:
$$
\begin{equation}
\Delta ([x_2^nx_1])=[x_2^nx_1]\otimes 1 +\sum _{k=0}^n\alpha _k^{(n)}g_2^k\, 
[x_2^{n-k}x_1]\otimes x_2^k.
\label{2coSer4} 
\end{equation}
If we define 
$$
\{x_2^nx_1\}={[x_2^nx_1]\over [n]_q! \cdot \prod_{s=0}^{n-1}(1-p_{12}p_{21}p_{22}^s)},
$$
then all coefficients desappear:
\begin{equation}
\Delta (\{x_2^nx_1\} )=\{x_2^nx_1\} \otimes 1 +\sum _{k=0}^n g_2^k\,  \{x_2^{n-k}x_1\} \otimes  \{x_2^k\} . 
\label{mon2}
\end{equation}
Below we show an alternative way how to develop these coproduct formulas using the shuffle representation.

\section{Shuffle representation}
The tensor space $T(W)$ of the linear space $W$ spand by 
the set of quantum variables $X=\{x_1, x_2,\ldots ,x_n$\}
has a structure of a braided Hopf algebra $Sh(W)$ with a braiding $\tau (u\otimes v)=p(v,u)^{-1}v\otimes u.$ 
We shall denote  the tensors $z_1\otimes z_2\otimes \ldots \otimes z_m,$ $z_i\in X$ considered as elements 
of $Sh(W)$  by $(z_1z_2 \cdots z_n)$ and call them {\it comonomials}. 
By definition the product on $T(W)$ is the so called {\it shuffle product}:
\begin{equation}
(u)(v)=\sum_{\substack{u=u_1\cdots u_\ell,\,v=v_1\cdots v_\ell\\1\leq\ell\leq\max\{\deg(u),\deg(v)\}}}\,
(\prod_{i<j} p_{v_iu_j}^{-1})\, (u_1v_1\cdots u_\ell v_\ell),
\label{prow}
\end{equation}
whereas the coproduct  $\Delta^b: Sh(W)\rightarrow Sh(W)\underline{\otimes } Sh(W)$ is defined by 
\begin{equation}
\Delta^b((u))=\sum_{u=u_1u_2}(u_1)\otimes(u_2).
\label{bcopro}
\end{equation}
The formula of the shuffle product is easier when one of the comonomials has the length one:
\begin{equation}
(w)(x_i)=\sum _{uv=w}p(x_i,v)^{-1}\cdot (ux_iv), \ \ (x_i)(w)=\sum _{uv=w}p(u,x_i)^{-1}\cdot (ux_iv).
\label{spro}
\end{equation}
From this equality, we deduce that
\begin{equation}
[(w), (x_i)]=\sum _{uv=w}(p(x_i,v)^{-1}-p(v,x_i))\cdot (ux_iv),
\label{proc}
\end{equation}
and
\begin{equation}
[(x_i),(w)]=\sum _{uv=w}(p(u,x_i)^{-1}-p(x_i,u))\cdot (ux_iv).
\label{proc1}
\end{equation}

The free algebra ${\bf k}\langle X\rangle $ considered as a subalgebra 
of $G\langle X\rangle $  becomes  a braided Hopf algebra if we define a braided coproduct $\Delta ^b$ as follows:
\begin{equation}
\Delta ^b(u)=\sum _{(u)}u^{(1)}\hbox{gr}(u^{(2)})^{-1}\underline{\otimes} u^{(2)}, \hbox{ where }\  
\Delta (u)=\sum _{(u)}u^{(1)}\otimes u^{(2)}.
\label{copro}
\end{equation}

The map $\Omega: x_i\rightarrow (x_i)$ defines a homomorphism of the braided Hopf algebra
${\bf k}\langle X\rangle $ into the braided Hopf algebra $Sh(W).$ 
If $p_{ij}$ are algebraically independent parameters, then  $\Omega $ is an isomorphism. Otherwise
the  kernel of $\Omega $ is  the largest Hopf ideal in ${\bf k}\langle X\rangle ^{(2)},$
where ${\bf k}\langle X\rangle ^{(2)}$ is the ideal of ${\bf k}\langle X\rangle ,$
generated by $x_ix_j,$ $1\leq i,j\leq n.$ The image of $\Omega $ is the so-called Nichols algebra 
of the braided space $W.$ See details in P. Schauenberg  \cite{Sch}, M. Rosso \cite{Ros},
M. Takeuchi \cite{Tak1},  D. Flores de Chela and J.A. Green  \cite{FC}, 
N. Andruskiewitsch, H.-J. Schneider  \cite{AS}.

\smallskip

The homomorphism $\Omega $  is 
extremely useful for calculating the coproduct due to  (\ref{copro}) and (\ref{bcopro}).
In this way, we may find alternative proof of the coproduct formulas (\ref{mon1}) and (\ref{mon2}).

\begin{lemma}
If $[n]_q!\neq 0,$ then
$$\Omega (\{x_2^n\})=q^{n(1-n)\over 2}(x_2^n), \ \ \  q=p_{22}.$$
Otherwise $\Omega (x_2^n)=0.$
\label{exm}
\end{lemma}
\begin{proof}
By induction on $n$ we may prove the explicit formula 
\begin{equation}
\Omega (x_2^n)=[n]_{q^{-1}}!(x_2^n).
\label{mo2}
\end{equation}
 Indeed, 
if $n=1,$ then there is nothing to prove. Using (\ref{spro}), we have
$$
\Omega (x_2^{n+1})=[n]_{q^{-1}}! (x_2^n)(x_2)=[n]_{q^{-1}}!(1+q^{-1}+\cdots + q^{-n}) (x_2^{n+1})=[n+1]_{q^{-1}}!(x_2^{n+1}).
$$
Because $[k]_{q^{-1}}=q^{-(k-1)}\cdot [k]_q,$ and $-1-2-\cdots -n={(n+1)(1-(n+1))\over 2}$, the lemma  is proven.
\end{proof}
\begin{lemma}
If $[k]_q\neq 0,$ $p_{12}p_{21}\neq q^{1-k},$ $q=p_{22},$ $1\leq k\leq n,$ then
$$\Omega (\{x_1x_2^n\})=p_{21}^{-n}q^{n(1-n)\over 2}(x_2^nx_1),$$
and
$$\Omega (\{x_2^nx_1\})=p_{12}^{-n}q^{n(1-n)\over 2}(x_1x_2^n).$$
Otherwise $\Omega ([x_1x_2^n])=\Omega ([x_2^nx_1])=0.$
\label{kmm}
\end{lemma}
\begin{proof}
By induction on $n$ we shall prove the explicit formulas
\begin{equation}
\Omega ([x_1x_2^n])=[n]_{q^{-1}}!\cdot p_{21}^{-n}\prod_{s=0}^{n-1} (1-p_{12}p_{21}q^{s})\cdot (x_2^nx_1),
\label{bic1}
\end{equation}
and 
\begin{equation}
\Omega ([x_2^nx_1])=[n]_{q^{-1}}!\cdot  p_{12}^{-n}\prod_{s=0}^{n-1} (1-p_{12}p_{21}q^{s})\cdot (x_1x_2^n).
\label{bic2}
\end{equation}
If $n=0,$ the equalities are evident. Using (\ref{proc}), we have
$$
[(x_2^nx_1),(x_2)]=\sum _{uv=x_2^nx_1}(p(x_2,v)^{-1}-p(v,x_2))\cdot (ux_2v)
$$
$$
=\sum_{k=0}^n(p_{21}^{-1}q^{-k}-p_{12}q^{k})\cdot (x_2^{n+1}x_1)=
(p_{21}^{-1} [n]_{q^{-1}}-p_{12} [n]_{q})\cdot (x_2^{n+1}x_1)
$$
$$
=(1-p_{12}p_{21}q^n)p_{21}^{-1} [n]_{q^{-1}} \cdot (x_2^{n+1}x_1),
$$
which completes the induction step. Similarly, using (\ref{proc1}), we have
$$
[(x_2),(x_1x_2^n)]=\sum _{uv=x_1x_2^n}(p(u,x_2)^{-1}-p(x_2,u))\cdot (ux_2v)
$$
$$
=\sum_{k=0}^n(p_{12}^{-1}q^{-k}-p_{21}q^{k})\cdot (x_1x_2^{n+1})=
(p_{12}^{-1} [n]_{q^{-1}}-p_{21} [n]_{q})\cdot (x_1x_2^{n+1})
$$
$$
=(1-p_{12}p_{21}q^n)p_{12}^{-1} [n]_{q^{-1}} \cdot (x_1x_2^{n+1}),
$$
which proves the second equality. Since  $[n]_{q^{-1}}!=q^{n(1-n)\over 2}\cdot [n]_q!,$
the lemma is proven.
\end{proof}

The coproduct formula  (\ref{mon2}) follows from the proven lemmas in the following way.
$$
(\Omega \otimes \Omega)\Delta^b(\{ x_2^nx_1\})
=\Delta^b (\Omega(\{ x_2^nx_1\}))
=p_{12}^{-n}q^{n(1-n)\over 2} \Delta^b ((x_1x_2^n))
$$
$$
=p_{12}^{-n}q^{n(1-n)\over 2}\left( 1\underline{\otimes } (x_1x_2^n)+\sum _{k=0}^n (x_1x_2^{n-k})\underline{\otimes } (x_2^{k})\right)
$$
$$
= 1 \underline{\otimes } \Omega \{ x_2^{n}x_1\}
+p_{12}^{-n}q^{n(1-n)\over 2} \sum _{k=0}^n p_{12}^{n-k} q^{(n-k)(n-k-1)\over 2}  q^{k(k-1)\over 2}
\Omega \{ x_2^{n-k}x_1\} \underline{\otimes }\Omega  \{ x_2^k\} =
$$
$$
=(\Omega \otimes \Omega)     \left(1\underline{\otimes }  \{ x_2^{n}x_1\} 
+\sum _{k=0}^n p_{12}^{-k} q^{k(k-n)}
\{ x_2^{n-k}x_1\} \underline{\otimes } \{ x_2^k\} \right) .
$$
Considering that $p_{ij}$ are algebraically independent parameters, we have
$$
\Delta^b(\{ x_2^nx_1\})=1\underline{\otimes }  \{ x_2^{n}x_1\} +\sum _{k=0}^n p_{12}^{-k} \cdot q^{k(k-n)}\cdot 
\{ x_2^{n-k}x_1\} \underline{\otimes } \{ x_2^k\}.
$$
Due to (\ref{copro}), we may write $u^{(1)}=u_b^{(1)}{\rm gr} (u^{(2)})$ and $u^{(2)}=u_b^{(2)}.$
We have 
${\rm gr} (\{ x_2^k\})=g_2^{k},$ and $\{ x_2^{n-k}x_1\}g_2^{k}= p_{12}^{k} \cdot q^{k(n-k)}
g_2^{k}\{ x_2^{n-k}x_1\}.$ Hence the component $u^{(1)}\otimes u^{(2)}$
is precisely $g_2^k\{ x_2^{n-k}x_1\} \otimes  \{ x_2^k\}.$ Of course, if the proven formula is valid for 
free parameters, it remans valid for arbitrary parameters provided that  $\{ x_2^k\} $  and $\{ x_2^{n-k}x_1\} ,$
$0\leq k\leq n$ are defined.
\begin{lemma}
If $p_{11}=q^3,$ $p_{22}=q,$ $p_{12}p_{21}=q^{-3}\neq 1,$ %$q\neq -1,$
then in the shuffle algebra the following decomposition is valid:
$$
(x_1x_2^3x_1)=p_{12}{q^2+q\over 1-q^3} \cdot (x_2x_1)(x_2^2x_1)
+q^2p_{12}^{2} {[4]_q-2\over 1-q^3} \cdot (x_2^2x_1)(x_2x_1)+q^3p_{12}^3\cdot (x_2^3x_1)(x_1).
$$
\label{leq}
\end{lemma}
\begin{proof}
We have
\begin{align*}
(x_2x_1)(\underline{x_2x_2x_1})&=\\
(x_2x_1\underline{x_2x_2x_1})&+
p_{21}^{-1}(x_2\underline{x_2}x_1\underline{x_2x_1})
+p_{21}^{-2}(x_2\underline{x_2x_2}x_1\underline{x_1})
+p_{21}^{-2}p_{11}^{-1}(x_2\underline{x_2x_2x_1}x_1)\\
+p_{21}^{-1}p_{22}^{-1}&(\underline{x_2}x_2x_1\underline{x_2x_1})
+p_{21}^{-2}p_{22}^{-1}(\underline{x_2}x_2\underline{x_2}x_1\underline{x_1})
+p_{21}^{-1}p_{11}^{-1}p_{22}^{-1}(\underline{x_2}x_2\underline{x_2x_1}x_1)\\
+p_{21}^{-2}p_{22}^{-2}&(\underline{x_2x_2}x_2x_1\underline{x_1})
+p_{21}^{-2}p_{11}^{-1}p_{22}^{-2}(\underline{x_2x_2}x_2\underline{x_1}x_1)\\
+p_{21}^{-2}p_{11}^{-1}&p_{22}^{-2}p_{12}^{-1}(\underline{x_2x_2x_1}x_2x_1)\\
&=(x_2x_1x_2^2x_1)+p_{21}^{-1}(1+p_{22}^{-1}+p_{22}^{-2})(x_2^2x_1x_2x_1)\\
&\quad\quad\quad\quad\quad\quad\quad\quad\quad+p_{21}^{-2}(1+p_{22}^{-1}+p_{22}^{-2})(1+p_{11}^{-1})(x_2^3x_1^2).
\end{align*}
\begin{align*}
(\underline{x_2x_2x_1})(x_2x_1)&=\\
(\underline{x_2x_2x_1}x_2x_1)+&p_{21}^{-1}(\underline{x_2x_2}x_2\underline{x_1}x_1)
+p_{21}^{-1}p_{22}^{-1}(\underline{x_2}x_2\underline{x_2x_1}x_1)
+p_{21}^{-1}p_{22}^{-2}(x_2\underline{x_2x_2x_1}x_1)\\
+p_{21}^{-1}p_{11}^{-1}(\underline{x_2x_2}&x_2x_1\underline{x_1})
+p_{21}^{-1}p_{22}^{-1}p_{11}^{-1}(\underline{x_2}x_2\underline{x_2}x_1\underline{x_1})
+p_{21}^{-1}p_{22}^{-2}p_{11}^{-1}(x_2\underline{x_2x_2}x_1\underline{x_1})\\
+p_{21}^{-1}p_{22}^{-1}&p_{11}^{-1}p_{12}^{-1}(\underline{x_2}x_2x_1\underline{x_2x_1})
+p_{21}^{-1}p_{22}^{-2}p_{11}^{-1}p_{12}^{-1}(x_2\underline{x_2}x_1\underline{x_2x_1})\\
+p_{21}^{-1}p_{22}^{-1}&p_{11}^{-1}p_{12}^{-2}(x_2x_1\underline{x_2x_2x_1})\\
&=p_{12}^{-1}p_{22}^{-2}(x_2x_1x_2^2x_1)+(1+p_{22}^{-1}+p_{22}^{-2})(x_2^2x_1x_2x_1)\\
&\quad\quad\quad\quad\quad\quad\quad\quad\quad+p_{21}^{-1}(1+p_{22}^{-1}+p_{22}^{-2})(1+p_{11}^{-1})(x_2^3x_1^2).
\end{align*}
\begin{align*}
(x_2x_2x_2x_1)(x_1)=&p_{11}^{-1}p_{12}^{-2}(x_2x_1x_2^2x_1)+p_{11}^{-1}p_{12}^{-1}(x_2^2x_1x_2x_1)\\
&\quad\quad\quad\quad\quad\quad\quad\quad\quad+(1+p_{11}^{-1})(x_2^3x_1^2)+p_{11}^{-1}p_{12}^{-3}(x_1x_2^3x_1).
\end{align*}

Taking into account relations $p_{11}=q^3,$ $p_{22}=q,$ $p_{21}=q^{-3}p_{12}^{-1},$ 
$1+p_{22}^{-1}+p_{22}^{-2}=q^{-2}[3]_q,$ we obtain 
$$
(x_2x_1)(x_2^2x_1)=\ \ \ \ \ \ \ \ \ (x_2x_1x_2^2x_1)+ q[3]_qp_{12}(x_2^2x_1x_2x_1)+ \ (q^3+1)q[3]_qp_{12}^2(x_2^3x_1^2)
$$
$$
(x_2^2x_1)(x_2x_1)=q^{-2}p_{12}^{-1}(x_2x_1x_2^2x_1)+  q^{-2}[3]_q(x_2^2x_1x_2x_1)+
(q^3+1)q^{-2}[3]_qp_{12}(x_2^3x_1^2) 
$$
$
(x_2^3x_1)(x_1)-q^{-3}p_{12}^{-3}(x_1x_2^3x_1)=
$
$$\ \ \ \ \ \ \ \ \ \ \ \ \ \ \ =\ q^{-3}p_{12}^{-2}(x_2x_1x_2^2x_1) +  
q^{-3}p_{12}^{-1}(x_2^2x_1x_2x_1) + \ \ \ \ \ (q^3+1)q^{-3}(x_2^3x_1^2).
$$
The determinant of the $3\times 3$ matrix of the  right-hand side coefficients is zero
because the last two columns are proportional with respect to the factor $(q^3+1)p_{12}.$
Therefore, the shuffles from the left-hand side are linearly dependent.
In which case the minors 
$$
\left|
\begin{matrix} 1 & q[3]_qp_{12} \cr 
q^{-2}p_{12}^{-1} &q^{-2}[3]_q
\end{matrix}
\right|= (q^{-2}-q^{-1})[3]_q; \ \ 
\left|
\begin{matrix} q^{-2}p_{12}^{-1} & q^{-2}[3]_q \cr 
q^{-3}p_{12}^{-2} & q^{-3}p_{12}^{-1}
\end{matrix}
\right|=-q^{-4}(1+q)p_{12}^{-2}
$$
are nonzero; that is, the first two shuffles are linearly independent as well as the second one and the last one are.
Thus, the last left-hand side shuffle is a linear combination of the first and second ones.
 Of course, it is easy to find the explicit values of the coefficients $\alpha $ and $\beta $ resolving
the system of equations 
$$
\left\{\begin{matrix} \alpha & + & q^{-2}p_{12}^{-1} \beta & = & q^{-3}p_{12}^{-2} \cr
q[3]_qp_{12} \alpha  & + & q^{-2}[3]_q \beta & = & q^{-3}p_{12}^{-1},
\end{matrix}
\right.
$$
$$
\alpha = -q^{-2}p_{12}^{-2}{q+1\over 1-q^3};  \ \  \ \ \beta =q^{-1}p_{12}^{-1} {2-[4]_q\over 1-q^3}.
$$
To find the coefficients of the required decomposition, it remains to multiply $\alpha $ and $\beta $ by $-q^{3}p_{12}^3.$
 %of $(x_1x_2^3x_1)$.
\end{proof}

\section{Coproduct formula}
By definition the quantum Borel algebra $U_q^+({\mathfrak g})$ related to the simple Lie algebra of type $G_2$
is a homomorphic image of $G\langle x_1, x_2\rangle $ subject to quantum Serre relations 
$$
[x_1x_2^4]=0,\ \ \ [x_1^2x_2]=0
$$
provided that the parameters of quantization satisfy
 $$
p_{11}=p_{12}^{-1}p_{21}^{-1}=p_{22}^3\neq 1, \ \ p_{22}\neq -1.
$$
Of course, this is  a two-parameter family of Hopf algebras.  As above, we put $q=p_{22}.$
Coproduct formulas  (\ref{2coSer}) with $n=4$ and (\ref{2coSer4}) with $n=2$ imply that the defining relations are skew-primitive
polynomials in $G\langle x_1, x_2\rangle .$ Therefore $U_q^+({\mathfrak g})$ keeps the Hopf algebra structure.

The subalgebra $A$ of $U_q^+({\mathfrak g})$ generated by $x_1,$ $x_2$ has a structure of braided
Hopf algebra if we define the braided coproduct by the same relation (\ref{copro}). Equation (\ref{bic1})
with $n=4$ and equation (\ref{bic2}) with $n=2$ show that $\Omega ([x_1x_2^4])=0$ and $\Omega ([x_1^2x_2])=0.$
Therefore the homomorphism $\Omega $ induces a homomorphism of braided Hopf algebras
$$ 
\bar\Omega : A \rightarrow Sh(W).
$$
It is well-known that if $q$ is not a root of unity, then this is an isomorphism.
We define a new element of $A$ as follows:
$$
\{ x_1x_2^3x_1\} \stackrel{df}{=} p_{21}^2{q^3+q^2\over 1-q^3}\{ x_1x_2\} \{ x_1x_2^2\}
+p_{21}{[4]_q-2\over 1-q^3}\{ x_1x_2^2\} \{ x_1x_2\}+\{ x_1x_2^3\}x_1.
$$
\begin{proposition} The ordered set 
$$
x_2<\{ x_1x_2^3\}<\{ x_1x_2^2\} < \{ x_1x_2^3x_1\}<\{ x_1x_2\}<x_1
$$
forms a set of PBW generators of the algebra $U_q^+({\frak g})$ over $G.$
\label{tion}
\end{proposition}
\begin{proof} 
In \cite{Pog} and independently in \cite{Ang}, it is shown that the ordered set 
$$
x_2<[x_1x_2^3]<[x_1x_2^2]< [[x_1x_2] ,[x_1x_2^2]]<[x_1x_2]<x_1
$$
forms a set of PBW generators of the algebra $U_q^+({\frak g})$ over $G.$
Of course, this implies that
$$
x_2<\{ x_1x_2^3\}<\{ x_1x_2^2\} <[\{ x_1x_2\} ,\{ x_1x_2^2\}]<\{ x_1x_2\}<x_1
$$
is also a set of PBW generators. By definition the element $\{ x_1x_2^3x_1\}$ has the form 
$$
\{ x_1x_2^3x_1\} = \alpha \{ x_1x_2\} \{ x_1x_2^2\}
+\beta \{ x_1x_2^2\} \{ x_1x_2\}+\{ x_1x_2^3\}x_1,
$$
where $\alpha \neq 0.$ Using evident formula $uv=[u,v]+p(u,v)vu,$ we obtain 
the decomposition of $\{ x_1x_2^3x_1\}$ in the above PBW basis:
$$
\alpha [\{ x_1x_2\} ,\{ x_1x_2^2\}] 
+\gamma \{ x_1x_2^2\} \{ x_1x_2\}+\{ x_1x_2^3\}x_1.
$$ In this decomposition $[\{ x_1x_2\} ,\{ x_1x_2^2\}]$ is the leading therm. Since $\alpha \neq 0,$
it follows that in the set of PBW generators we may replace $[\{ x_1x_2\} ,\{ x_1x_2^2\}]$
with $\{ x_1x_2^3x_1\}.$
\end{proof}

\begin{theorem} The following coproduct formula is valid
$$
\Delta (\{ x_1x_2^3x_1\})=\{ x_1x_2^3x_1\} \otimes 1 +g_1^2g_2^3\otimes \{ x_1x_2^3x_1\} 
+\sum _{k=0}^3 g_1g_2^k\,  \{x_2^{3-k}x_1\} \otimes  \{x_1x_2^k\} .
$$
\label{c5}
\end{theorem}
\begin{proof} 
By Lemma \ref{kmm}, we have 
$$
\bar\Omega (\{ x_1x_2\} )=p_{21}^{-1}(x_2x_1), \ \ 
\bar\Omega (\{ x_1x_2^2\} )=p_{21}^{-2}q^{-1}(x_2^2x_1), \ \ 
\bar\Omega (\{ x_1x_2^3\} )=p_{21}^{-3}q^{-3}(x_2^3x_1).
$$
Using these equalities and Lemma \ref{leq}, we obtain 
$$
\bar\Omega(\{ x_1x_2^3x_1\} )=
p_{21}^2{q^3+q^2\over 1-q^3}\cdot p_{21}^{-1} \cdot p_{21}^{-2}q^{-1} (x_2x_1)(x_2^2x_1)
$$
$$
+p_{21}{[4]_q-2\over 1-q^3} \cdot p_{21}^{-1} \cdot p_{21}^{-2}q^{-1} (x_2^2x_1)(x_2x_1)+
p_{21}^{-3}q^{-3} (x_2^3x_1)(x_1)
%$$
%$$
=q^3(x_1x_2^3x_1)
$$
because $p_{21}^{-1}q^{-1}=q^3\cdot p_{12},$ $p_{21}^{-2}q^{-1}=q^3\cdot q^2p_{12}^2,$
and $p_{21}^{-3}q^{-3}=q^3\cdot q^3p_{12}^3.$
We have,
$$
(\bar\Omega \underline{\otimes } \bar\Omega)\Delta^b(\{ x_1x_2^3x_1\})=
\Delta^b (\bar\Omega(\{ x_1x_2^3x_1\} ))
=\Delta^b (q^3(x_1x_2^3x_1))
$$
$$
=q^3(x_1x_2^3x_1)\underline{\otimes }1+1\underline{\otimes }q^3(x_1x_2^3x_1)
+\sum _{k=0}^3 q^3(x_1x_2^{3-k}) \otimes  (x_2^kx_1).
$$
 Lemma \ref{kmm} implies 
$$
(x_1x_2^{3-k})=p_{12}^{3-k}q^{(3-k)(3-k-1)\over 2}\bar\Omega( \{x_2^{3-k}x_1\}), \ \ 
(x_2^kx_1)=p_{21}^kq^{k(k-1)\over 2}\bar\Omega( \{x_1x_2^k\}).
$$ 
By this reason, the tensor under the sum equals
$$
p_{12}^{3-k}p_{21}^kq^{6-3k+k^2}\cdot  \bar\Omega \underline{\otimes }\bar\Omega
 (\{x_2^{3-k}x_1\} \underline{\otimes }  \{x_1x_2^k\}).
$$
If $q$ is a free parameter, then $\bar\Omega $ is an isomorphism, and in $A,$ we have
$$
\Delta^b (\{ x_1x_2^3x_1\})=\{ x_1x_2^3x_1\} \otimes 1 +1\otimes \{ x_1x_2^3x_1\} 
+\sum _{k=0}^3 \tau_k g_1g_2^k\,  \{x_2^{3-k}x_1\} \otimes  \{x_1x_2^k\} ,
$$
where $\tau_k=p_{12}^{3-k}p_{21}^kq^{6-3k+k^2}.$
Due to (\ref{copro}), we may write $u^{(1)}=u_b^{(1)}{\rm gr} (u^{(2)})$ and $u^{(2)}=u_b^{(2)}.$
Since
${\rm gr} ( \{x_1x_2^k\})=g_1g_2^{k},$ and 
$$
\{x_2^{3-k}x_1\}g_1g_2^{k}= \mu_k \cdot g_1g_2^{k}\{x_2^{3-k}x_1\}
$$
with $\mu_k=p_{11}p_{12}^kp_{21}^{3-k}p_{22}^{(3-k)k},$ 
it follows that the component $u^{(1)}\otimes u^{(2)}$
equals $\tau_k\mu_k g_1g_2^k\,  \{x_2^{3-k}x_1\} \otimes  \{x_1x_2^k\} .$
It remains to check that $\tau_k\mu_k=1$:
$$
p_{12}^{3-k}p_{21}^kq^{6-3k+k^2}\cdot q^3p_{12}^kp_{21}^{3-k}q^{(3-k)k}
=q^9\cdot (p_{12}p_{21})^{3-k} (p_{21}p_{12})^k=q^9 q^{-3(3-k)}q^{-3k}=1.
$$
\end{proof}

\end{document}